\documentclass[11pt]{amsart}
\usepackage[margin=1in]{geometry}

\usepackage{amssymb}
\usepackage{amsthm}
\usepackage{amsmath}
\usepackage{mathrsfs}
\usepackage{amsbsy}
\usepackage[all]{xy}
\usepackage{bm}
\usepackage{hyperref}
\usepackage{tikz}
\usepackage{array}
\usepackage{float}
\usepackage{enumerate}
\usepackage{xcolor}
\usepackage{hhline}
\allowdisplaybreaks
\usepackage[noadjust]{cite}

\usepackage{caption}
\usepackage{subcaption}
\usepackage{tabu}
\usepackage{diagbox}
\usepackage{tikz}
\usepackage{bbm}
\usepackage{booktabs}

\DeclareFontFamily{U}{rcjhbltx}{}
\DeclareFontShape{U}{rcjhbltx}{m}{n}{<->rcjhbltx}{}
\DeclareSymbolFont{hebrewletters}{U}{rcjhbltx}{m}{n}

\let\aleph\relax\let\beth\relax

\DeclareMathSymbol{\aleph}{\mathord}{hebrewletters}{39}
\DeclareMathSymbol{\beth}{\mathord}{hebrewletters}{98}

\usepackage[noabbrev,capitalise]{cleveref}

\newenvironment{enumerate*}%
  {\begin{enumerate}[(I)]%
    \setlength{\itemsep}{10pt}%
    \setlength{\parskip}{0pt}}%
  {\end{enumerate}}

\newtheorem{theorem}{Theorem}[section]

\newtheorem{conjecture}[theorem]{Conjecture}

\theoremstyle{definition}

\title{Rearranging small sets for distinct partial sums}

\author[]{Noah Kravitz}
\address[]{Department of Mathematics, Princeton University, Princeton, NJ 08540, USA}
\email{nkravitz@princeton.edu}

\begin{document}
\maketitle

\begin{abstract}
A conjecture of Graham (repeated by Erd\H{o}s) asserts that for any set $A \subseteq \mathbb{F}_p \setminus \{0\}$, there is an ordering $a_1, \ldots, a_{|A|}$ of the elements of $A$ such that the partial sums $a_1, a_1+a_2, \ldots, a_1+a_2+\cdots+a_{|A|}$ are all distinct.  We give a very short proof of this conjecture for sets $A$ of size at most $\log p/\log\log p$.
\end{abstract}

\section{A conjecture of Graham}

For $A$ a finite subset of an abelian group, say that an ordering $a_1, \ldots, a_{|A|}$ of the elements of $A$ is \emph{valid} if the partial sums $a_1, a_1+a_2, \ldots, a_1+a_2+\cdots+a_{|A|}$ are all distinct.  The following striking conjecture first appeared in a 1971 open problem list of Ronald Graham \cite{graham} and was later repeated in a book of Erd\H{o}s and Graham \cite{EG} (see also \cite[Problem \#475]{Bloom}).

\begin{conjecture}[\cite{graham}]\label{conj:main}
Let $p$ be a prime.  Then every subset $A \subseteq \mathbb{F}_p \setminus\{0\}$ has a valid ordering.
\end{conjecture}

Although many papers have been written about this conjecture and related problems, the state of the art is still essentially that Conjecture~\ref{conj:main} holds when $|A| \leq 12$ (see \cite{CP} and the references therein) and when $A$ is a non-zero sum set of size $p-2$ or $p-3$ (see \cite{HOS} and the references therein).  Many of the arguments for small $A$ use the Polynomial Method and rely on Alon's Combinatorial Nullstellensatz.

Alspach (as attributed in \cite{BH}) independently posed a very similar conjecture for finite cyclic groups, and versions in other groups (both abelian and nonabelian) have been studied; for further history and more extensive references, see the recent papers \cite{CDF,long}.  We also mention that this line of inquiry is related to combinatorial designs and the Hall--Paige Conjecture, as described in \cite{MP}.

In this short note, we make modest partial progress towards Conjecture~\ref{conj:main} by showing that it holds for small sets $A$; the novelty is that our bound $\log p/\log\log p$ tends to infinity with $p$.

\begin{theorem}\label{thm:main}
Let $p$ be a prime.  Then every subset $A \subseteq \mathbb{F}_p \setminus\{0\}$ of size
$$ |A| \leq \frac{\log p}{\log \log p}$$
has a valid ordering.
\end{theorem}

We prove this theorem by applying a ``rectification'' result of of Lev \cite{Lev} (refining work of Bilu, Lev, and Ruzsa \cite{BLR}) and then establishing the ``integer version'' of Conjecture~\ref{conj:main}, as follows.

\begin{theorem}\label{thm:integers}
Every finite subset $A \subseteq \mathbb{Z} \setminus\{0\}$ has a valid ordering.
\end{theorem}

We later learned that Will Sawin \cite{will} proved a very similar result, using the same two main steps, in a MathOverflow post in 2015.  His argument and ours differ in the details of both steps, however, and we believe that it is useful to have both approaches recorded in the literature.

\section{Proofs}

We begin with the integer version of Conjecture~\ref{conj:main}.

\begin{proof}[Proof of Theorem~\ref{thm:integers}]
We will prove the stronger statement that there is a valid ordering of $A$ in which all of the positive elements appear before all of the negative elements.  Let $P,N$ be sets of positive integers such that $A=P \cup (-N)$.  It suffices to find orderings $p_1, \ldots, p_{|P|}$ of $P$ and $n_1, \ldots, n_{|N|}$ of $N$ such that
$$p_1+\cdots+p_i \neq n_1+\cdots+n_j \quad \text{unless $(i,j) \in \{(0,0),(|P|,|N|)\}$};$$
then the ordering
$$p_{|P|}, p_{|P|-1}, \ldots, p_1, -n_1, -n_2, \ldots, -n_{|N|}$$
of $A$ is valid.

We proceed by induction on $|A|$, where the base case $|A|=0$ is trivial.  Now suppose that $|A| \geq 1$.  Without loss of generality, we have $$\sum_{p \in P} p \geq \sum_{n \in N} n.$$
The desired conclusion is obvious if $|P|< 2$, so suppose that $|P| \geq 2$.  Then there is some $p^* \in P$ such that
$$\sum_{p \in P \setminus \{p^*\}} p \neq \sum_{n \in N} n;$$
let $P':=P \setminus \{p^*\}$.  The induction hypothesis provides orderings $p_1,...,p_{|P|-1}$ of $P'$ and $n_1,...,n_{|N|}$ of $N$ such that
$$p_1+\cdots+p_i \neq n_1+\cdots+n_j \quad \text{unless $(i,j) \in \{(0,0),(|P|-1,|N|)\}$}.$$
Our choice of $p^*$ ensures that we also have
$$p_1+\cdots+p_{|P|-1} \neq n_1+\cdots+n_{|N|}.$$
Now, the orderings $p_1,...,p_{|P|-1},p^*$ of $P$ and $n_1,...,n_{|N|}$ of $N$ are as desired.
\end{proof}

We turn next to the rectification result that we will use to deduce Theorem~\ref{thm:main} from Theorem~\ref{thm:integers}.  Let $A,B$ be subsets of (possibly different) abelian groups.  We say that a bijection $f: A \to B$ is an \emph{$\ell$-Freiman isomorphism} if
$$x_1+\cdots+x_\ell=y_1+\cdots+y_\ell \Longleftrightarrow f(x_1)+\cdots+f(x_\ell)=f(y_1)+\cdots+f(y_\ell)$$
for all $x_1, \ldots, x_\ell, y_1, \ldots, y_\ell \in A$ (allowing repetitions).  If $0 \in A$ and $B$ has no nonzero elements of order dividing $\ell$ (as when, for instance, the group containing $B$ is torsion-free), then every $\ell$-Freiman isomorphism $f$ satisfies $f(0)=0$; in this case, we conclude that $f$ is also a $k$-Freiman isomorphism for all $k<\ell$.  Bilu, Lev, and Ruzsa \cite{BLR} used the Pigeonhole Principle to show that small subsets of $\mathbb{F}_p$ are always Freiman-isomorphic, with high order, to sets of integers.  We will use the following (optimal) refinement due to Lev \cite{Lev}.

\begin{theorem}[\cite{Lev}]\label{thm:rectification}
Let $\ell \in \mathbb{N}$, let $p$ be a prime, and let $A \subseteq \mathbb{F}_p$.  If $|A| \leq \lceil\log p/\log \ell\rceil$, then $A$ is $\ell$-Freiman-isomorphic to a set of integers.
\end{theorem}

See \cite{GR} for further discussion of rectification principles in additive combinatorics.

We are now ready to prove Theorem~\ref{thm:main}.

\begin{proof}[Proof of Theorem~\ref{thm:main}]
Let $A':=A \cup \{0\}$.  Then $|A'| \leq \lceil\log p/\log(|A|-1)\rceil$, and Theorem~\ref{thm:rectification} provides an $(|A|-1)$-Freiman isomorphism $f$ from $A'$ to some set $B$ of integers.  Note that $0 \notin f(A)$, and that every valid ordering of $f(A)$ pulls back to a valid ordering of $A$ (since the conditions for an ordering to be valid can be described in terms of non-equalities of sums with length at most $|A|-1$).  The result now follows from Theorem~\ref{thm:integers}.
\end{proof}

\section{Remarks}

\begin{enumerate}
    \item The valid orderings constructed in our proofs of Theorems~\ref{thm:integers} and~\ref{thm:main} are ``two-sided'' in the sense that their reverses are also valid; this condition is equivalent to the absence of zero-sum proper consecutive suborderings.  In the setting of Conjecture~\ref{conj:main} (and its generalization to other abelian groups), is it always possible to find two-sided valid orderings?
    \item Say that an abelian group $G$ is \emph{sequenceable} (respectively, \emph{strongly sequenceable}) if every subset of $G \setminus \{0\}$ has a valid (respectively, two-sided valid) ordering.  The argument of Theorem~\ref{thm:integers} can be easily modified to show the stronger statement that an abelian group $H$ is strongly sequenceable if and only if $H \times \mathbb{Z}$ is strongly sequenceable.  (Theorem~\ref{thm:integers} corresponds to the case where $H$ is the trivial group.)  One can modify the proof of Theorem~\ref{thm:integers} as follows: Partition $A=P \cup Z \cup (-N)$ where $P,N$ are sets of elements with positive second coordinate and $Z$ is a set of elements with second coordinate $0$.  If $\sum_{z \in Z}z \neq 0$, then consider orderings of $A$ consisting of the elements of $P$, then the elements of $Z$, then the elements of $-N$.  If $\sum_{z \in Z}z=0$, then pick some suitable $z^* \in Z$ and consider orderings of $A$ consisting of the elements of $P$, then the elements of $Z \setminus \{z^*\}$, then the elements of $-N$, with $z^*$ placed at either the very beginning or the very end.  It could be interesting to formulate versions of this principle in nonabelian settings (see \cite{CDFO}).\footnote{Note added in revision: Such results for the infinite dihedral group and other semidirect products with $\mathbb{Z}$ have recently appeared in \cite{CDF2}.}
    \item By replacing Theorem~\ref{thm:rectification} with Lev's more general rectification criterion for abelian groups \cite{Lev}, we can extend Theorem~\ref{thm:main} to abelian groups with no elements of small torsion: If $G$ is an abelian group with no nonzero elements of order strictly smaller than $p$, then every subset $A \subseteq G \setminus \{0\}$ of size at most $\log p/\log\log p$ has a valid ordering.  (See the discussion in \cite{long,CP} for previous results in this direction.)
    \item It is known (see, e.g., \cite{GR}, following \cite{Freiman}) that even moderate-sized subsets of $\mathbb{F}_p$ can be rectified if one adds a small-doubling assumption; again, our arguments apply to this scenario.
    \item Combining the second and third remarks, we see that if $H_1$ is a strongly sequenceable abelian group and $H_2$ is an abelian group with no nonzero elements of order strictly smaller than $p$, then every subset $A \subseteq (H_1 \times H_2) \setminus\{(0,0)\}$ with $|\pi_{H_2}(A)| \leq \log p/\log\log p$ has a two-sided valid ordering.  In the context of finite cyclic groups, this significantly extends the results of \cite{long} and some of the results of \cite{CDFO}.
\end{enumerate}

\section*{Acknowledgements}
The author was supported in part by the NSF Graduate Research Fellowship Program under grant
DGE–203965.  I thank Stefan Steinerberger for bringing Graham's conjecture to my attention and for engaging in helpful discussions.  I thank Simone Costa for giving useful comments on a draft of this paper, and I thank Raphael Steiner for bringing the reference \cite{will} to my attention.

\end{document}